\documentclass[11pt]{article}

\usepackage{amsmath, amsthm, amsopn, amssymb, enumerate,hyperref}

\newtheorem{thm}{Theorem}
\newtheorem{lemma}[thm]{Lemma}
\newtheorem{cor}[thm]{Corollary}
\newtheorem{claim}[thm]{Claim}

\numberwithin{equation}{section}
\numberwithin{thm}{section}

\theoremstyle{definition}
\newtheorem{remark}[thm]{Remark}

\newcommand{\vol}{\mathrm{vol}}
\newcommand{\Z}{\mathbb{Z}}
\newcommand{\R}{\mathbb{R}}

\title{Grounded Lipschitz functions on trees are typically flat}

\author{
  Ron Peled\thanks{School of Mathematical Sciences, Tel Aviv University, Tel Aviv, Israel. E-mail address: {\tt peledron@post.tau.ac.il}. Supported by an ISF grant and an IRG grant.}
  \and
  Wojciech Samotij\thanks{School of Mathematical Sciences, Tel Aviv University, Tel Aviv, Israel; and Trinity College, Cambridge~CB2~1TQ, UK. E-mail address: {\tt ws299@cam.ac.uk}. Research supported in part by ERC Advanced Grant DMMCA, a~Trinity College JRF, and a grant from the Israel Science Foundation}
  \and
  Amir Yehudayoff\thanks{Department of Mathematics, Technion-IIT, Haifa, Israel. E-mail address: {\tt amir.yehudayoff@gmail.com}. Horev fellow -- supported by the Taub Foundation. Research supported by grants from ISF and BSF.}
}

\begin{document}

\date{}

\maketitle

\begin{abstract}
  A grounded $M$-Lipschitz function on a rooted $d$-ary tree is an integer-valued map on the vertices that changes by at most $M$ along edges and attains the value zero on the leaves. We study the behavior of such functions, specifically, their typical value at the root $v_0$ of the tree. We prove that the probability that the value of a uniformly chosen random function at $v_0$ is more than $M+t$ is doubly-exponentially small in $t$. We also show a similar bound for continuous (real-valued) grounded Lipschitz functions.
\end{abstract}

\section{Introduction}

This note studies the typical behavior of grounded Lipschitz
functions on trees. For an integer $M \ge 1$, call an integer-valued
function $f$ on the vertices of a graph \emph{$M$-Lipschitz} if
$|f(u) - f(v)| \le M$ for every two adjacent vertices $u$ and $v$.
We consider the rooted $d$-ary tree of depth $k$, that is, the tree
whose each non-leaf vertex, including the root, has $d$ children.
We denote this tree by $T(d,k)$ and its root by $v_0$.
Let $L_M(d,k)$ be the set of all $M$-Lipschitz functions on $T(d,k)$
which take the value zero on all leaves. For example, $L_M(d,1)$ has
$2M+1$ elements, one element for each of the possible values of
$f(v_0)$ between $-M$ and $M$.

The behavior of $M$-Lipschitz functions on trees and on regular
expander graphs was studied in~\cite{PSY}, where ``flatness'' of typical
$M$-Lipschitz functions was proved for relatively small values of
$M$ (depending on the degree and expansion properties of the graph).
For trees, the result of \cite{PSY} states that if
\begin{equation}\label{eq:MLip_on_expander_tree_cond}
M \le c d / \log d
\end{equation}
for some small absolute constant $c$, then a typical element of
$L_M(d,k)$ is very flat in the sense that on
all but a tiny fraction of the vertices it takes values between $-M$ and $M$.
Precisely, it was shown
there (see~\cite[Theorem~1.7]{PSY}) that under the assumption
\eqref{eq:MLip_on_expander_tree_cond}, if we let $f$ be a uniformly
chosen random function in $L_M(d,k)$ then for every integer $s\ge 1$
and vertex $v\in T(d,k)$,
\begin{equation*}
  \Pr(|f(v)| > sM) \le \exp\left(-\frac{d(d-1)^{s-1}}{5(M+1)}\right).
\end{equation*}

In this work we study further the distribution of the value at the root vertex $v_0$
for a uniformly chosen function in $L_M(d,k)$. Our main result is
that this value is very tightly concentrated around $0$ for
\emph{every} $M$ and $d \ge 2$, regardless of whether the
assumption~\eqref{eq:MLip_on_expander_tree_cond} is satisfied.

\begin{thm}
  \label{thm:main}
  Let $d \geq 2$ and $M,k \geq 1$. If $f$ is chosen uniformly at random from $L_M(d, k)$, then for every integer $t \geq 1$,
  \begin{equation}
    \label{eq:main}
    \Pr( f(v_0) = M+t ) \le (9/10)^{d^{\lfloor (t-1)/M \rfloor}} \cdot \Pr( f(v_0) = 0).
  \end{equation}
  Moreover, if $d \ge 3$, then the constant $9/10$ above may be replaced by $(3/4)^d$.
  In addition,
  \[
  \Pr ( f(v_0) = 0 ) \leq \frac{1}{1+2^{1-d}M}.
  \]
\end{thm}

Of course, by symmetry, the theorem implies a corresponding bound on $\Pr( f(v_0) = -M-t )$.
We prove Theorem~\ref{thm:main} by induction on $k$. The argument has three
steps. The first step establishes that $p(t) := \Pr(f(v_0) = t)$ is
unimodal in $t$ with maximum at $t=0$. The second step shows that
$p(t)$ decays at least exponentially in $t/M$, i.e., that $p(t+M)
\leq 9 p(t)/10$ for $t \geq 1$. In the third step,
inequality~\eqref{eq:main} is derived by induction on $t$.

As a second result, we conclude that Theorem~\ref{thm:main} remains valid when we let $f$ be a uniformly chosen \emph{continuous} (i.e., real-valued) grounded Lipschitz function. Formally, we call a real-valued function $f$ on the vertices of a graph \emph{Lipschitz} if $|f(u) - f(v)| \le 1$ for every pair $u, v$ of adjacent vertices. Let $L_\infty(d,k)$ be the family of all such Lipschitz functions on $T(d,k)$ that take the value zero on all leaves.

\begin{thm}
  \label{thm:mainCon}
  Let $d \geq 2$ and $k \geq 1$.   If $f$ is chosen uniformly at random from   $L_\infty(d,k)$, then for every $x > 0$,
  \[
  \Pr( f(v_0) \geq 1+x ) \leq 2^{d+2} \cdot (9/10)^{d^{\lceil x \rceil - 1}}.
  \]
  Moreover, if $d \ge 3$, then the constant $9/10$ above may be replaced by $(3/4)^d$.
\end{thm}

We also mention briefly a related model, the model of \emph{random
graph homomorphisms}. An integer-valued function $f$ on the vertices
of a graph is called a graph homomorphism (or a homomorphism
height function) if $|f(u) - f(v)| = 1$ for every pair $u, v$ of
adjacent vertices. The study of the typical properties of a
uniformly chosen random graph homomorphism was initiated in
\cite{BHM} and results were subsequently obtained for:\ tree-like
graphs~\cite{BHM}, the hypercube~\cite{Gal, Kah2}, and (the
nearest-neighbor graph on finite boxes in) the $d$-dimensional
integer lattice $\Z^d$ for large $d$~\cite{Pel}. A lower bound on the typical range of random graph homomorphisms on general graphs was established in~\cite{BenYadYeh}. Graph
homomorphisms are similar to $1$-Lipschitz functions (see the Yadin
bijection described in~\cite{Pel} for a precise connection) and a
result analogous to our Theorem~\ref{thm:main} was proved for them
in~\cite{BHM}, using a similar, though somewhat simpler, method.

We end the introduction with a discussion of the role of the
assumption~\eqref{eq:MLip_on_expander_tree_cond}. In a forthcoming
paper \cite{Hhom} we will show that under this assumption, more is
true of a uniformly chosen function from $L_M(d,k)$. In fact, at a
vertex which is at even distance from the leaves, the function will
take the value $0$ with high probability, namely, with probability
at least $1 - 2\exp(-c d/M)$ for some absolute constant $c>0$. That
is, the function value is concentrated on a single number. Note that
this implies that for a vertex at odd distance from the leaves, with
a similarly high probability, the function value at all of its
neighbors is zero, and that conditioned on this event, the value at
the vertex is uniform on $\{-M,\ldots, M\}$.

We expect that such a strong concentration of the values of the
random function fails when $M\gg d$. More precisely, let us fix
$M,d$ such that $M\gg d$ and denote by $\mathcal{L}_k$ the
distribution of the value at the root for a tree of depth $k$. We
believe that $\mathcal{L}_k$ is no longer concentrated on a single
value when $k$ is even. Moreover, we suspect that $\mathcal{L}_k$
has a limit as $k$ tends to infinity (so that there is
asymptotically no distinction between even and odd depths). It would
be interesting to establish such a transition phenomenon between the
cases $M\ll d$ and $M\gg d$.

\section{From unimodality to doubly exponential decrease}

In this section, we prove Theorems~\ref{thm:main} and~\ref{thm:mainCon}. Fix integers $d \geq 2$ and $M \geq 1$. For integers $t$ and $k \geq 1$, we let
\[
G(t,k) = |\{f \in L_M(d,k) \colon f(v_0) = t\}|.
\]
Several times in our proofs, we will use the fact that $G(t,k) = G(-t,k)$ for every $t$ and $k$, which follows by symmetry.

\subsection{A recursive formula}

Since for every $k > 1$, the children of the root of $T(d,k)$ can be regarded as roots of isomorphic copies of $T(d,k-1)$, we have
\begin{equation}
  \label{eq:Gtk-recursion}
  G(t,k) = \sum_{t-M \leq t_1,\ldots,t_d \leq t+M} \prod_{s=1}^d G(t_s,k-1)
  = \left( \sum_{i = -M}^M G(t+i,k-1) \right)^d.
\end{equation}

\subsection{Unimodality}

The following claim establishes unimodality.

\begin{claim}
  \label{clm:first}
  If $k \geq 1$ and $t \geq 0$, then $G(t+1,k) \leq G(t,k)$.
\end{claim}

\begin{proof}
  We prove the claim by induction on $k$. If $k=1$, then $G(t,k)=1$ if $0 \le t \leq M$ and $G(t,k) = 0$ if $t > M$.
  Assume that $k > 1$ and $t \ge 0$. By \eqref{eq:Gtk-recursion}, it suffices to show that
  \begin{equation}
    \label{eq:first-suff}
    G(t+M+1,k-1) \le G(t-M,k-1).
  \end{equation}
  To see this, we consider two cases. First, if $t-M \ge 0$, then \eqref{eq:first-suff} follows directly from the inductive assumption. Otherwise, we note that $t+M+1 \ge M-t > 0$ and hence by symmetry and induction,
  \[
  G(t+M+1,k-1) \le G(M-t,k-1) = G(t-M,k-1).\qedhere
  \]
\end{proof}

\subsection{Exponential decay of $G(t,k)$}

The following claim establishes exponential decay of $G(t,k)$. We start with the case $d \geq 3$. The case $d=2$ is more elaborate and we handle it separately later on.

\begin{lemma}
  \label{lem:withAlpha}
  If $d \geq 3$ and $k,t \geq 1$, then $G(t+M,k) \leq (3/4)^d \cdot G(t,k)$.
\end{lemma}

\begin{proof}
  Fix some $d \ge 3$. We prove the lemma by induction on $k$. Suppose that $t \ge 1$. If $k = 1$, then $G(t+M,k) = 0$ and the claimed inequality holds vacuously. Assume that $k > 1$. To simplify the notation, we let
  \[
  \alpha = (3/4)^d \qquad \text{and} \qquad G(s) = G(s,k-1) \quad \text{for each $s \in \Z$}.
  \]
  Moreover, for a set $S \subseteq \Z$, we let
  \[
  G(S) = \sum_{s \in S} G(s).
  \]
  If $t > M$, then by the inductive assumption,  $G(t+M+i) \leq \alpha G(t+i)$ for every $-M \leq i \leq M$, so \eqref{eq:Gtk-recursion} implies that
  \[
  G(t+M,k) = G( \{t, \ldots, t+2M \} )^d \leq (\alpha \cdot G(\{t-M, \ldots, t+M\}))^d = \alpha^d \cdot G(t,k).
  \]
  Assume that $1 \leq t \leq M$, let
  \[
  A = G(\{0, \ldots, t-1\}), \quad
  B = G(\{1, \ldots, M-t\}), \quad \text{and} \quad
  C = G(\{t, \ldots, t+M\}),
  \]
  and observe that, by~\eqref{eq:Gtk-recursion} and symmetry,
  \[
  G(t,k) = (A+B+C)^d.
  \]
  On the other hand, since
  \begin{equation}
    \label{eq:tM-tMM-cover}
    \{t + M + 1, \ldots, t + 2M\} = (\{1, \ldots, t\} + 2M) \cup (\{t+1, \ldots, M\} + M),
  \end{equation}
  then \eqref{eq:Gtk-recursion}, the inductive assumption, and Claim~\ref{clm:first} imply that
  \[
  G(t+M, k) \leq \big(\alpha^2 G(\{1, \ldots, t\}) +  \alpha G(\{t+1, \ldots, M\}) + C\big)^d \leq (\alpha^2 A + \alpha B + C)^d.
  \]
  Moreover, Claim~\ref{clm:first} implies
  \[
  (M+1)(A+B) \geq M C \quad \text{and} \quad M A \geq (A+B).
  \]
  It therefore follows that
  \begin{align*}
    \frac{G(t+M,k)}{G(t,k)} &
    \leq \left(\frac{\alpha^2 A + \alpha B + C}{A+B+C}\right)^d
    \leq \left(\frac{\alpha^2 A + \alpha B + (1 + 1/M)(A+B)}{A+B+(1+1/M)(A+B)}\right)^d \\
    & = \left(\frac{(1+\alpha+1/M)(A+B) - (\alpha - \alpha^2)A}{(2+1/M)(A+B)}\right)^d \\
    & \leq \left(\frac{1+\alpha+(1-\alpha+\alpha^2)/M}{2+1/M}\right)^d \leq \left( \max\left\{ \frac{1+\alpha}{2}, \frac{2+\alpha^2}{3} \right\} \right)^d
    \leq \alpha,
  \end{align*}
  where the final inequality holds by our assumption that $d \ge 3$.
\end{proof}

\begin{remark}
  \label{remark:alpha}
  It follows from the proof of Lemma~\ref{lem:withAlpha} that even when $d = 2$, the statement of the lemma still holds as long as we replace the constant $(3/4)^d$ with some constant $\alpha < 1$ that satisfies
  \begin{equation}
    \label{eq:alpha}
    \left(\frac{1+\alpha+(1-\alpha+\alpha^2)/M}{2+1/M}\right)^d \le \alpha.
  \end{equation}
  Unfortunately, the smallest solution to~\eqref{eq:alpha} tends to $1$ as $M \to \infty$. We thus need a more careful analysis to handle the case $d=2$. Our proof in the case $d=2$ shall require a mild lower bound on $M$. We therefore note that if $M \le 10$ and $\alpha = 9/10$, then \eqref{eq:alpha} is satisfied.
\end{remark}

\begin{lemma}
  \label{lem:d=2}
  Suppose that $d = 2$. For all $k, t \ge 1$, we have $G(t+M,k) \le (9/10) \cdot G(t,k)$.
\end{lemma}
\begin{proof}
  By Remark~\ref{remark:alpha}, we can safely assume that $M \ge 11$. Choose the following parameters:
  \[
  \alpha = 9/10, \quad \mu = 1/4, \quad \beta = 1/3 \quad \text{and} \quad m = \lceil \mu M \rceil.
  \]
  We are going to prove the following stronger statement by induction on $k$:
  \begin{equation}
    \label{eq:GtMk-ind}
    G(t+M,k) \leq
    \begin{cases}
      \alpha \cdot G(t,k) & \text{if $t \in \{1, \ldots, m-1\}$}, \\
      \alpha^2 \cdot G(t,k) & \text{if $t \geq m$}.
    \end{cases}
  \end{equation}
  If $k = 1$, then~\eqref{eq:GtMk-ind} holds vacuously as $G(t+M,k) = 0$ for every $t \geq 1$.
  Assume that $k > 1$ and fix some $t \ge 1$. To simplify notation, for every $s \in \Z$, we let
  \[
  G(s) = G(s,k-1)
  \]
  and for a set $S \subseteq \Z$,
  \[
  G(S) = \sum_{s \in S} G(s).
  \]
  If $t > M$, then by the inductive assumption and \eqref{eq:Gtk-recursion}, we have
  \[
  G(t+M,k) = G( \{t, \ldots, t+2M \} )^2 \leq (\alpha \cdot G(\{t-M, \ldots, t+M\}))^2 = \alpha^2 \cdot G(t,k).
  \]
  Assume that $1 \leq t \leq M$, let
  \[
  A = G(\{0, \ldots, t-1\}), \quad
  B = G(\{1, \ldots, M-t\})  \quad \text{and} \quad
  C = G(\{t, \ldots, t+M\}),
  \]
  and observe that by~\eqref{eq:Gtk-recursion} and symmetry,
  \[
  G(t,k) = (A+B+C)^2.
  \]
  We split the proof into two cases, depending on the value of $t$.

  \medskip
  \noindent
  {\bf Case 1: $\mathbf{t \ge m}$.}
  Identity~\eqref{eq:Gtk-recursion}, the inductive assumption, and Claim~\ref{clm:first} imply that (recall~\eqref{eq:tM-tMM-cover})
  \[
  \begin{split}
    G(t+M,k)^{1/2} & \le \alpha^2 \cdot \alpha \cdot G(\{1, \ldots, t\}) + \alpha^2 \cdot G(\{t+1, \ldots, M\}) + C \\
    & \le \left(\frac{t}{M} \cdot \alpha^3 + \frac{M-t}{M} \cdot \alpha^2 \right) \cdot (A+B) + C \\
    & \le (\mu \alpha^3 + (1-\mu)\alpha^2) \cdot (A+B) + C
  \end{split}
  \]
  and Claim~\ref{clm:first} implies that
  \begin{equation}
    \label{eq:ABC-two}
    (M+1)(A+B) \geq M C .
  \end{equation}
  It hence follows that
  \[
  \begin{split}
    \frac{G(t+M,k)}{G(t,k)} & \leq \left(\frac{(\mu \alpha^3 + (1-\mu)\alpha^2) \cdot (A+B) + C}{A+B+C}\right)^2 \\
    & \leq \left(\frac{1+\mu\alpha^3 + (1-\mu)\alpha^2 + 1/M}{2 + 1/M}\right)^2 <  \alpha^2,
  \end{split}
  \]
  where in the last inequality we used the assumption that $M \ge 11$.

  \medskip
  \noindent
  {\bf Case 2: $\mathbf{t < m}$.}
  Identity~\eqref{eq:Gtk-recursion}, the inductive assumption, and Claim~\ref{clm:first} imply that (recall~\eqref{eq:tM-tMM-cover})
  \begin{equation}
    \label{eq:shift-bound-2}
    G(t+M,k)^{1/2} \le
    \alpha^3 \cdot G(\{1, \ldots, t\}) + \alpha \cdot G(\{t+1, \ldots, m-1\}) + \alpha^2 \cdot G(\{m, \ldots, M\}) + C.
  \end{equation}
  We now further split into two cases, depending on whether or not the following inequality is satisfied:
  \begin{equation}
    \label{eq:split}
    G(\{m, \ldots, M\}) \ge \beta \cdot G(\{1, \ldots, M\}).
  \end{equation}

  If~\eqref{eq:split} holds, then by~\eqref{eq:shift-bound-2} and Claim~\ref{clm:first},
  \[
  \begin{split}
    G(t+M,k)^{1/2} & \le (\beta\alpha^2 + (1-\beta)\alpha) \cdot G(\{1, \ldots, M\}) + C \\
    & \le (\beta\alpha^2 + (1-\beta)\alpha) \cdot (A+B) + C.
  \end{split}
  \]
  Consequently, by~\eqref{eq:ABC-two},
  \[
  \begin{split}
    \frac{G(t+M,k)}{G(t,k)} & \leq \left(\frac{(\beta\alpha^2 + (1-\beta)\alpha) \cdot (A+B) + C}{A+B+C}\right)^2 \\
    & \leq \left(\frac{1+\beta\alpha^2 + (1-\beta)\alpha + 1/M}{2 + 1/M}\right)^2 < \alpha,
  \end{split}
  \]
  where in the last inequality we again used the assumption that $M \ge 11$.

  If~\eqref{eq:split} does not hold, then we let
  \[
  D = G(\{t, \ldots, m-1\}) \quad \text{and} \quad E = G(\{1, \ldots, M\}) .
  \]
  Identity~\eqref{eq:Gtk-recursion}, Claim~\ref{clm:first}, and the converse of~\eqref{eq:split} imply
  \[
  \begin{split}
    G(t+M,k)^{1/2} & = D + G(\{m, \ldots, 2M+t\}) \le D + \frac{2M+t-m+1}{M-m+1} \cdot G(\{m, \ldots, M\}) \\
    & < D + \frac{2M}{M-m+1} \cdot \beta \cdot E \le D + \frac{2}{1 - \mu} \cdot
    \beta \cdot E.
  \end{split}
  \]
  On the other hand, again by symmetry and Claim~\ref{clm:first}, we have that
  \[
  G(t,k)^{1/2} = G(\{t-M,\ldots,t+M\}) \geq D + E .
  \]
  So, as $D \le E$,
  \[
  \frac{G(t+M,k)}{G(t,k)}
  \le \left(1 - \frac{1-\mu-2\beta}{1-\mu} \cdot \frac{E}{D+E}\right)^2  \le \left(1 - \frac{1-\mu-2\beta}{2-2\mu}\right)^2
  \alpha. \qedhere
  \]
\end{proof}

\subsection{The full bound}

The exponential decay established in Lemmas~\ref{lem:withAlpha} and~\ref{lem:d=2} easily implies our main theorems.

\begin{cor}
  \label{cor:full-bound}
  If $k,t \geq 1$, then
  \begin{equation}
    \label{eq:GtpMk}
    G(t+M,k) \leq \alpha^{d^{\lfloor (t-1)/M \rfloor}} G(t,k),
  \end{equation}
  where $\alpha = 9/10$ if $d = 2$ and $\alpha = (3/4)^d$ if $d \ge 3$.
\end{cor}

\begin{proof}
  We prove the statement by induction on $k$. If $k=1$, then $G(t+M,k)=0$. Assume that $k>1$. When $1 \leq t \leq M$, we have $\lfloor (t-1)/M \rfloor =0$, and~\eqref{eq:GtpMk} follows directly from Lemmas~\ref{lem:withAlpha} and~\ref{lem:d=2}. If $t \geq M+1$, then by the inductive assumption and \eqref{eq:Gtk-recursion} we have
  \[
  \begin{split}
    G(t+M,k) & = \big( G(t,k-1) + \ldots + G(t+2M,k-1) \big)^d \\
    & \leq   \big( \alpha^{d^{\lfloor (t-M-1)/M \rfloor}} (G(t-M,k-1) + \ldots + G(t+M,k-1) ) \big)^d \\
    & = \alpha^{d^{1+\lfloor (t-M-1)/M \rfloor}} G(t,k) = \alpha^{d^{\lfloor (t-1)/M \rfloor}} G(t,k).
    \qedhere
  \end{split}
\]
\end{proof}

\begin{proof}[Proof of Theorem~\ref{thm:main}]
  The first part of the theorem follows immediately from the corollary above. The upper bound on $\Pr(f(v_0) = 0)$ follows from the inequality $G(M,k) \geq 2^{-d} G(0,k)$, which we prove below, together with symmetry and Claim~\ref{clm:first}. If $k = 1$, we have $G(M,k) = G(0,k) = 1$. If $k > 1$, identity \eqref{eq:Gtk-recursion} and symmetry imply that
  \[
  \left( \frac{G(M,k)}{G(0,k)} \right)^{1/d}
  \geq \frac{1}{2}.
  \qedhere
  \]
\end{proof}

\begin{proof}[Proof of Theorem~\ref{thm:mainCon}]
  Let $f$ be a uniformly chosen random element of $L_\infty(d,k)$ and let $x > 0$. The claimed bound on the probability that $f(v_0)$ exceeds $1+x$ follows fairly easily from Theorem~\ref{thm:main}. To see this, let, for every positive integer $M$, $f_M$ be a uniformly chosen random element of $L_M(d,k)$. A moment of thought reveals that the sequence $f_M/M$ converges to $f$ in distribution.

  Indeed, letting $V$ be the set of internal (non-leaf) vertices of $T(d,k)$, one may naturally view $L_\infty(d,k)$ as a convex polytope $P \subseteq \R^V$. Let $\mu$ and $\mu_M$ be the distributions of $f$ and $f_M/M$, respectively. Observe that $\mu = \lambda / \vol(P)$, where $\lambda$ is the $|V|$-dimensional Lebesgue measure, and that $\mu_M$ is the uniform measure on the (finite) set $P \cap (\frac{1}{M}\Z)^V$. Since $P$ is compact (as clearly $P \subseteq [-k,k]^V$), every continuous function $g \colon P \to \R$ is uniformly continuous and therefore,
  \[
  \lim_{M \to \infty} \int_P g \, d\mu_M = \int_P g \, d\mu.
  \]

  Now, Theorem~\ref{thm:main} implies that for $M \geq 1/x$, letting $\alpha$ be as in the statement of Corollary~\ref{cor:full-bound},
  \[
  \begin{split}
    \Pr \left( \frac{f_M(v_0)}{M} \geq 1+x \right) & = \sum_{t \ge xM} \Pr(f_M(v_0) = M+t) \le \Pr(f_M(v_0) = 0) \cdot \sum_{s = \lfloor x-1/M \rfloor}^\infty M\alpha^{d^s} \\
    & \leq \frac{M}{1+2^{1-d}M} \cdot  \alpha^{d^{\lfloor x-1/M \rfloor}} \cdot \left(1 + \sum_{r = 1}^\infty \alpha^{d^r/2} \right) \le \frac{8M}{1+2^{1-d}M} \cdot \alpha^{d^{\lfloor x-1/M \rfloor}},
  \end{split}
  \]
  where in the last inequality we used the fact that $\alpha^{d^r/2} \le (9/10)^{2^{r-1}}$, and hence
  \[
  \Pr(f(v_0) \geq 1+x) = \lim_{M \to \infty} \Pr \left( \frac{f_M(v_0)}{M} \geq 1+x \right) \leq 2^{d+2} \cdot \alpha^{d^{\lceil x \rceil - 1}}.\qedhere
  \]
\end{proof}

\bibliographystyle{amsplain}
\bibliography{LipOnTree}

\end{document}